\documentclass[12pt,a4paper]{article}
\usepackage[utf8]{inputenc}
\usepackage{amsmath}
\usepackage{amsfonts}
\usepackage{amssymb}
\usepackage{graphicx}
\usepackage{amsthm}
\usepackage[left=2cm,right=3cm,top=4cm,bottom=3cm]{geometry}

\usepackage{cite}

\newtheorem{theorem}{Theorem}
\newtheorem{lemma}{Lemma}
\newtheorem{proposition}{Proposition}

\newtheorem{definition}{Definition}
\newtheorem{example}{Example}

\newtheorem{remark}{Remark}

\title{On Subspace Convex-Cyclic Operators}
\author{Dilan Ahmed$^{2}$ \and Mudhafar Hama$^{3}$ \and Jaros{\l}aw Wo\'{z}niak$^1$ \and Karwan Jwamer$^3$}
\date{%
\small{${^1}$ Institute of Mathematics, Department of Mathematics and Physics, University of Szczecin, ul. Wielkopolska 15, 70-451 Szczecin, Poland; \\ \texttt{wozniak@univ.szczecin.pl} \\
$^{2}$ University of Sulaimani, College of Education, Department of Mathematics, Kurdistan Region, Sulaimani, Iraq; \\\texttt{ dilan.ahmed@univsul.edu.iq}\\
$^{3}$ University of Sulaimani, College of Science, Department of Mathematics, Kurdistan Region, Sulaimani, Iraq;\\ \texttt{mudhafar.hama@univsul.edu.iq} \quad and  \quad \texttt{karwan.jwamer@univsul.edu.iq} \\
}}

\begin{document}
\maketitle
\abstract{Let $\mathcal{H}$ be an infinite dimensional real or complex separable Hilbert space. We introduce a special type of a bounded linear operator $T$ and its important relation with invariant subspace problem on $\mathcal{H}$: operator $T$ is said to be is subspace convex-cyclic for a subspace $\mathcal{M}$, if there exists a vector whose orbit under $T$ intersects the subspace $\mathcal{M}$ in a relatively dense set. We give the sufficient condition for a subspace convex-cyclic transitive operator $T$ to be subspace convex-cyclic. We also give a special type of Kitai criterion related to invariant subspaces which implies subspace convex-cyclicity. We conclude showing a counterexample of a subspace convex-cyclic operator which is not subspace convex-cyclic transitive.}

\textbf{keywords:} {ergodic dynamical systems, convex-cyclic operators; Kitai criterion; convex-cyclic transitive operators}


\section{Introduction}
Ergodic dynamical systems seem to be of interest for a few decades, with an increasing number of papers appearing lately (see, for example, \cite{ergo1,ergo2,ergo3,ergo4,ergo5}), a large number of them concerning convex-cyclic operators.

A bounded linear operator $T$ on an infinite  dimensional separable Hilbert space is convex-cyclic(see \cite{A}) if there exists a vector $x$ in $\mathcal{H}$ such that $\widehat{Orb(T,x)} = \{ P(T)x : P  \, \text{is a} \text{ convex polynomial} \} $  is dense in $\mathcal{H}$ and the vector $x$ is said to be convex-cyclic vector for $T$. A bounded linear operator $T$ is said to be cyclic if there exists a vector $x$ in $\mathcal{H}$ such that the linear span of the orbit $[T,x]=span \{T^nx: n \in \mathbb{N} \} $ is dense in $\mathcal{H}$ and $x$ is called cyclic vector. If the orbit $Orb(T,x)= \{T^nx: n \in \mathbb{N} \}$ itself is dense  in $\mathcal{H}$ without of linear span, then $T$ is called hypercyclic and $x$ is called hypercyclic vector. The operator $T$ is said to be supercyclic if the cone generated by     $Orb(T,x)$ i.e.$\mathbb{C}Orb(T,x)= \{\lambda T^nx: \lambda \in \mathbb{C} \, \text{and} \, n \in \mathbb{N} \}$ is dense in $\mathcal{H}$ and $x$ is called supercyclic vector\cite{4}, \cite{3}. 
In \cite{C} it is mentioned that between a set and its linear span there is a convex hull, from this we get that every hypercyclic operator is convex-cyclic and every convex-cyclic operator is cyclic.

In this work we want to modify the notions given above and introduce the new concept of subspace convex-cyclic operator. In our case the orbit of subspace hypercyclic, subspace supercyclic and subspace convex-cyclic under hypercyclic, supercyclic and convex-cyclic operator respectively, intersected with a given selected subspace is dense in that subspace. 

The paper  is organized as follows. First we define the concept of subspace convex-cyclic operators and construct an example of subspace convex-cyclic operator which does not need to be convex-cyclic operator(Example \ref{example: 1st}). Next we prove that being subspace convex-cyclic transitive implies being subspace convex-cyclic operator. Then we show that a ``subspace convex-cyclic criterion" holds. To this end we find an interesting relation between our new operator and invariant subspaces. We also show, by giving a proper example, that for being subspace convex-cyclic being transitive and fulfilling the criterion are not necessary conditions. In the end we present some open questions concerning subspace cyclic operators.

\section{Definition and Examples}
Let $\mathcal{H}$ be an infinite dimensional real or complex separable Hilbert space. Whenever we talk about a subspace $\mathcal{M}$ of $\mathcal{H}$ we will assume that $\mathcal{M}$ is closed topologically. And let $ \textbf{B}(\mathcal{H})$ be the algebra of all linear bounded operators on $\mathcal{H}$. 
We start with our main definition.

\begin{definition}
Let $ T \in \textbf{B}(\mathcal{H})$ and let $\mathcal{M}$ be a non-zero subspace of $\mathcal{H}.$ We say that $T$ is \textbf{subspace convex-cyclic} operator, if there exist $ x \in \mathcal{H}$ such that $ \widehat{Orb(T,x)} \cap \mathcal{M} $ is dense in $\mathcal{M}$, where 
\begin{align*}
\widehat{Orb(T,x)} & = \left\lbrace P(T) x : P\, \text{is convex polynomial} \, \right\rbrace, \quad i.e.  \\
& = \left\lbrace P(T)x: P(T):= a_0 + a_1 T + a_2 T^2 + \cdots + a_n T^n, n \in \mathbb{N}, \sum \limits_{i=0}^{n} a_i = 1  \right\rbrace.
\end{align*}
Such a vector $x$ is said to be a subspace convex-cyclic vector. 
\end{definition}

We will write $ \mathcal{M} $ convex-cyclic  instead of subspace convex-cyclic. Moreover, let us define $CoC(T,\mathcal{M}) := \{x\in \mathcal{H} : \widehat{Orb(T,x)} \cap \mathcal{M}  \text{\, is dense in } \mathcal{M}  \}$ \, as the set of all subspace convex-cyclic vectors for $\mathcal{M}$.

\begin{remark}
Note that $\mathcal{M}$ can be any non-empty subset,  convex or not.
\end{remark}

\begin{example}
\label{example: 1st}
Let $T$ be a convex-cyclic operator on $\mathcal{H}$ and $I$ be the identity operator on $\mathcal{H}$. Then $T \oplus I : \mathcal{H} \oplus \mathcal{H}  \to  \mathcal{H} \oplus \mathcal{H}  $ is subspace convex-cyclic operator for subspace $\mathcal{M}= \mathcal{H} \oplus \{0\}$ with subspace convex-cyclic vector $x \oplus 0.$ \\
In fact, since $T$ is convex-cyclic operator on $\mathcal{H}$, so there exist $x \in \mathcal{H}$, such that $ \widehat{Orb(T,x)}$ is dense in $\mathcal{H}$.
Now we can consider the $T \oplus I : \mathcal{H} \oplus \mathcal{H}  \to  \mathcal{H} \oplus \mathcal{H}.$ \\
Let $\mathcal{M}:=\mathcal{H} \oplus \{ 0 \} \subseteq \mathcal{H} $ and there exist $ m := x \oplus  0 $ such that
\begin{align*}
\widehat{Orb}(T \oplus I,m ) & = \left\lbrace P(T \oplus I)m : P\; \text{is convex polynomial} \right\rbrace \\
& = \left\lbrace P(T)x \oplus 0 : P\; \text{is  convex polynomial} \right\rbrace \\
& \subseteq \mathcal{H} \oplus \{0\} = \mathcal{M}.   
\end{align*}
And since $\widehat{Orb(T,x)}$ is dense in $\mathcal{H}$, then $\widehat{Orb}(T \oplus I,(x \oplus  0) ) \cap \mathcal{H} \oplus \{0\}  $ is dense in $\mathcal{H} \oplus \{0\} =\mathcal{M}$
so we get that $ T \oplus I$ is a subspace convex-cyclic operator.
\end{example}

\begin{remark}
\label{Rem:JD}
The above example shows that if the operator $T$ is subspace convex-cyclic then $T$ does not need to be convex-cyclic. For clarifing that  let us recall the following Propositions \ref{Pro: 1} and \ref{Pro: 2.25Book2}, from \cite{A} and \cite{B2} respectively.

\end{remark}

\begin{proposition}[see \cite{A}]
\label{Pro: 1}
Let $T:X \to X$ be an operator. If $T$ is convex-cyclic, then
\begin{enumerate}
\item
$||T|| > 1$,
\item
$ \sup \{ ||T^n||: n \geq 1 \}= + \infty$,
\item
$ \sup \{ ||T^{*n} \Lambda ||: n \in \mathbb{N} \}= + \infty, \:$ for every $\Lambda \neq 0 \: in \: X^*$.
\end{enumerate}
\end{proposition}

\begin{proposition}[see \cite{B2}]
\label{Pro: 2.25Book2}
Let $ S:X \to X$ and $ T:Y \to Y$ be operators. If $ S \oplus T $ is hypercyclic then so are $ S $ and $ T $.
\end{proposition}
As we mentioned before, every hypercyclic operator is convex-cyclic, so  Proposition \ref{Pro: 2.25Book2} is of a great usage here.
\begin{remark}
Clearly $T \oplus I$ is not convex-cyclic operator. In fact, assume that $T \oplus I$ is convex-cyclic on $ \mathcal{H} \oplus \mathcal{H},$ then  by Proposition \ref{Pro: 2.25Book2} the identity operator must be convex-cyclic on $\{ 0 \}$,   which is impossible, because the norm of identity operator is equal to one, and by Proposition \ref{Pro: 1} we get a contradiction.
\end{remark}

\section{Subspace Convex-Cyclic Transitive Operators}
In this section we define subspace convex-cyclic transitive operators, and we will show that they will be subspace convex-cyclic operators.
First we state the classical equivalence of topological transitivity\cite{B1} and \cite{B2}, also convex-cyclic operators \cite{C}. 
 
\begin{definition}
Let $T \in \textbf{B}(\mathcal{H})$ and let $\mathcal{M}$ be a non-zero subspace of $\mathcal{H}$. We say that $T$ is $\mathcal{M}$ convex-cyclic transitive with respect to $\mathcal{M}$ if for all non-empty sets $U \subset \mathcal{M}$ and $V \subset \mathcal{M}$, both are relatively open, there exist a convex polynomial $P$ such that  $ U \cap P(T)(V) \neq \phi $ \, or \, $P(T)^{-1}(U) \cap V \neq  \phi $   contains a relatively open non-empty subset of $\mathcal{M}$.
\end{definition}

We use the ideas from \cite{B1, B2,3} changing them to work for convex polynomial spans and generalizing them, obtaining the following.

\begin{theorem}
\label{theorem:1}
Let $T \in \textbf{B}(\mathcal{H})$ and let $\mathcal{M}$ be non-zero subspace of $\mathcal{H}$. Then 
$$CoC(T,\mathcal{M}) = \bigcap\limits_{j=1}^{\infty} \: \bigcup\limits_{P \in \mathcal{P}} P(T)^{-1}(\mathcal{B}_j),$$
where $\mathcal{P}$ is the collection of all convex polynomials and $\{\mathcal{B}_j\}$ is a countable open basis for relative topology $\mathcal{M}$ as a subspace of $\mathcal{H}$.
\end{theorem}

\begin{proof}
Observe that $x \in \bigcap\limits_{j=1}^{\infty} \bigcup\limits_{P \in \mathcal{P}} P(T)^{-1}(B_j)$, if and only if, for all $j \in \mathbb{N} $, there exist a convex polynomials $P$ such that $x \in P(T)^{-1}(B_j)$ which implies  $P(T)(x) \in B_j$. But since $\{ B_j \}$ is a basis for the relatively topology of $\mathcal{M}$, this occurs if and only if $\widehat{Orb(T,x)} \cap \mathcal{M}$ is dense in $\mathcal{M}$, that is $x \in CoC(T,\mathcal{M})$.
\end{proof}

\begin{lemma}
\label{lemma:1}
Let $T \in \textbf{B}(\mathcal{H})$ and let $\mathcal{M}$ be a non-zero subspace of $\mathcal{H}$. Then the following are equivalent:
\begin{enumerate}
\item
$T$ is $\mathcal{M}$ convex-cyclic transitive with respect to $\mathcal{M}$.
\item
for each relatively open subsets $U \; and \; V \;of \; \mathcal{M}$, there exist $P \in \mathcal{P}$ such that $P(T)^{-1}(U) \cap V$ is relatively open subset in $\mathcal{M}$. Where $\mathcal{P}$ is the set of all convex polynomials.
\item
for each relatively open subsets $U \; and \; V \;of \; \mathcal{M}$, there exist $P \in \mathcal{P}$ such that $P(T)^{-1}(U) \cap V \neq \phi$ and $P(T)(\mathcal{M}) \subset \mathcal{M}$.
\end{enumerate}
\end{lemma}

\begin{proof}
$(3) \Rightarrow (2)$ Since $P(T): \mathcal{M} \to \mathcal{M}$ is continuous and we know that $U$ is relatively open in $\mathcal{M}$, then
$P(T)^{-1}(U)$ is also relatively open in $\mathcal{M}$. Now, if we take any $V \vert_{open} \subset \mathcal{M}$, then let $W := P(T)^{-1}(U) \cap V$, which is open and $W \subset \mathcal{M}.$ \\

$(2) \Rightarrow (1)$  Since for each relatively open subsets $U$ and $V$, \;$P(T)^{-1}(U) \cap V$ is relatively open subset in $\mathcal{M}$, so 
$P(T)^{-1}(U) \cap V \neq \phi;$ now let $W := P(T)^{-1}(U) \cap V,$ then  $ \; W \vert_{open} \subset \mathcal{M}.$ \\

$(1) \Rightarrow (3)$ By definition of $\mathcal{M}$ convex-cyclic transitive, there exist $U$ and $V$ relatively open subsets in $\mathcal{M}$ such that $W:=P(T)^{-1}(U) \cap V \neq \phi, \;$and this set $ W $ is relatively open in $\mathcal{M}$, and $W \subset P(T)^{-1}(U)$. Then $P(T)(W) \subset U \; and \; U\subset \mathcal{M}$, so we get that

\begin{equation*}
 P(T)(W) \subset  \mathcal{M}. 
\end{equation*}
Let $x \in \mathcal{M}$, we must show that $P(T)(\mathcal{M}) \subset \mathcal{M}$.
Take $w_0 \in W,$ since $W$ is relatively open in $\mathcal{M}$ and $x \in \mathcal{M}$ so there exist $r > 0$ such that 
$w_0 + rx \in W$. \\
But $P(T)(W)\subseteq \mathcal{M},$ that is 
$P(T)(w_0 +rx)= P(T)(w_0)+rP(T)x \in \mathcal{M}$, so $P(T)(w_0) \in \mathcal{M}$ and 
$\mathcal{M}$ is subspace thus
$r^{-1}\left( -P(T)(w_0) + P(T)(w_0)+rP(T)(x) \right) \in \mathcal{M} $, that is $P(T)(x) \in \mathcal{M}$. This is true for any $x \in \mathcal{M},$ hence for $P(T)(x) \in \mathcal{M}$, that is $P(T)(\mathcal{M}) \subseteq \mathcal{M}.$
\end{proof}

\begin{theorem}
\label{theorem:2}
Let $T \in \textbf{B}(\mathcal{H})$ and let $\mathcal{M}$ be non-zero subspace of $\mathcal{H}$. If $T$ is $\mathcal{M}$ convex-cyclic transitive, then $T$ is $\mathcal{M}$ convex-cyclic.
\end{theorem}
\begin{proof}
It is a direct consequence of proofs of Lemma \ref{lemma:1} and Theorem \ref{theorem:1}
\end{proof}

\begin{remark}
It is natural to ask if the converse of Theorem \ref{theorem:2} is true or not. We will answer this question later in Proposition \ref{proposition: 4.8}.
\end{remark}

\section{Subspace Convex-Cyclic Criterion}
This section is devoted to introduce a type of Kitai's criterion \cite{Int3}, which is a sufficient criterion for an operator to be $\mathcal{M}$ convex-cyclic. Also we will relate it with invariant subspaces and we will see that the converse of Theorem \ref{theorem:2} in general is not true.

\begin{theorem}
\label{th: cri1}
Let $T \in \textbf{B}(\mathcal{H})$ and let $\mathcal{M}$ be a non-zero subspace of $\mathcal{H}$. Assume that there exist $X$ and $Y$, dense subsets of $\mathcal{M}$ 
such that for every $x \in X$ and $y \in Y$ there exist a sequence $\{P_k\}_{k \geq 1}$ of convex polynomials such that
\begin{enumerate}
\item
$P_k(T)x \to 0,\: \forall x \in X$,
\item
for each $y \in Y$, there exists a sequence $\{x_k \}$ in $\mathcal{M}$ such that
$x_k \to 0$  and $ P_k(T)x_k \to y,$
\item
$\mathcal{M}$ is an invariant subspace for $P_k(T)$ for all $k \geq 0 $.
\end{enumerate}
Then $T$ is $\mathcal{M}$ convex-cyclic operator.
\end{theorem}

\begin{proof}
To prove that $T$ is $\mathcal{M}$ convex-cyclic operator we will use Lemma \ref{lemma:1} and Theorem \ref{theorem:2}. Let $ U $ and $V$ be non-empty relatively open subsets of $\mathcal{M}$. We will show that there exists $k \geq 0 $ such that $P_k(T)(U) \cap V \neq \phi$. Since $X$ and $Y$ are dense in $\mathcal{M}$, there exists $v \in X \cap V$ and $ u \in Y \cap U$. Furthermore, since $U$ and $V$ are relatively open,
there exists $ \epsilon > 0$ such that the $\mathcal{M}$-ball centered at $v$ of radius $\epsilon$ is contained in $V$ and the $\mathcal{M}$-ball centered at $u$ of radius $\epsilon$ is contained in $U$.
By hypothesis, given these $ v \in X$ and $u \in Y$ , one can choose $k$ large enough such that there exists $x_k \in \mathcal{M}$ with \\
$\Vert P_k(T) v \Vert < \frac{\epsilon}{2}$, \quad $\Vert x_k \Vert < \epsilon$ \quad and \quad $ \Vert P_k(T) x_k - u \Vert < \frac{\epsilon}{2}.$ We have:
\begin{enumerate}
\item
Since $v \in \mathcal{M}$ and $x_k \in \mathcal{M}$, it follows that $v+x_k \in \mathcal{M}$. Also, since 
$$ \Vert (v + x_k) - v \Vert  =  \Vert x_k \Vert  < \epsilon ,$$
it follows that $v+x_k $ is in $\mathcal{M} - ball$ centered at $v$ of radius $\epsilon$ and hence $v+x_k \in V.$
\item
Since $v$ and $x_k$ are in $\mathcal{M}$ and $\mathcal{M}$ is invariant under $P_k(T),$ it follows that $P_k(T)(v+x_k) \in \mathcal{M}.$ Also
$$ \Vert P_k(T)(v+x_k) - u \Vert \leq \Vert   P_k(T)(v)\Vert + \Vert P_k(T)(x_k) - u \Vert < \frac{\epsilon}{2} +  \frac{\epsilon}{2} = \epsilon,$$
and hence\\
 $ P_k(T)(v+x_k) $ is in the $\mathcal{M} - ball$ centered at $u$ of radius $\epsilon$ and thus $ P_k(T)(v+x_k) \in U$.
\end{enumerate}
So by steps (1) and (2), $T$ is $\mathcal{M}$ convex-cyclic transitive and by Theorem \ref{theorem:2} we get that $v+x_k \in P_k(T)^{-1}(U) \cap V$,
that is $ P_k(T)^{-1}(U) \cap V \neq \phi$ which means $T$ is $\mathcal{M}$ convex-cyclic operator.
\end{proof}

As we mentioned before, most of the papers related to the topic of convex-cyclic operators depend on the Kitai's \cite{Int3} criterion, which does not contain an analogue of our condition 3. We clarify the need of this additional condition in details in Example \ref{example: on cri2}.

\begin{theorem}
\label{th: cri2}
Let $T \in \textbf{B}(\mathcal{H})$ and let $\mathcal{M}$ be a non-zero subspace of $\mathcal{H}$. Assume there exist $X$ and $Y$, 
subsets of $\mathcal{M}$ where just $Y$ is dense in $\mathcal{M}$ such that for every $x \in X$ and $y \in Y$ there exist a sequence $\{P_k\}_{k \geq 0}$ of convex polynomials such that
\begin{enumerate}
\item
$P_k(T)x \to 0,\: \forall x \in X$
\item
for each $y \in Y$, there exists a sequence $\{x_k \}$ in $\mathcal{M}$ such that
$x_k \to 0$  and $ P_k(T)x_k \to y,$
\item
$X \subset \bigcap \limits_{k=1}^{\infty} P_k(T)^{-1}(\mathcal{M})$.
\end{enumerate}
Then $T$ is $\mathcal{M}$ convex-cyclic operator.
\end{theorem}

\begin{proof}
We will use the idea from \cite{4}, namely let $\{ \xi_j \}^{\infty}_{j=1}$ be a sequence of positive number such that
$$\lim \limits_{ j \to \infty} \left( j \xi_j + \sum^{\infty}_{i=j+1} \xi_i \right)=0. $$
In fact, from condition (1) for all $ \lambda_j > 0, \quad \Vert P_{k_j}(T)(x) \Vert < \lambda_j $, \\
and from condition (2) for all $ \epsilon_j > 0 \quad \Vert P_{k_j}(T)(x_j)-y_j \Vert < \epsilon_j $. \\
So, we can define a sequence of positive numbers $\{ \xi_j \}^{\infty}_{j=1}$ as follows: \\
$\xi_i = \lambda_i \quad for \; i=1,2, \cdots , j.$ And $\xi_i = \epsilon_i \quad for \; i = j+1 , \cdots $ such that \\
$$\lim \limits_{ j \to \infty} \left( j \xi_j + \sum^{\infty}_{i=j+1} \xi_i \right)=0.$$
Since $\mathcal{H}$ is separable, we can assume that $Y= \{ y_j \}^{\infty}_{j=1} $ for some sequence $ \{ y_j \}^{\infty}_{j=1}.$ We can construct a sequence $ \{ x_j \}^{\infty}_{j=1} \subset X $ and $ \{ k_j \}^{\infty}_{j=1} $ of $ \{ k \}^{\infty}_{j=1} $ by induction. Let $x_1 \in X$ and $k_1$ be such that 
$\Vert x_1 \Vert + \Vert P_{k_1} (T)(x_1) - y_1  \Vert < \xi_1. $ for each $j$ choose $k_j$ and $x_j \in X $ such that  
$\Vert x_j \Vert + \Vert P_{k_j} (T)(x_i) \Vert  + \Vert P_{k_i}(T)(x_j) \Vert + \Vert P_{k_j}(T)(x_j) - y_j \Vert < \xi_j$   \quad for all $i <j. $ 
Since
$\sum \limits_{i=1}^{\infty} \Vert x_i \Vert <  \sum \limits_{i=1}^{\infty} \xi_i ,$ we can let 
 $x= \sum \limits_{i=1}^{\infty} x_i,$ and $x$ is well defined.\\
 From condition (3) for every $j$ we have $P_{k_j} (T)(x) \in \mathcal{M}$ and
\begin{align*}
\Vert P_{k_j} (T)(x) - y_j \Vert  & = \Vert P_{k_j} (T)(x_j) - y_j  + \sum \limits_{i=1}^{j-1}P_{k_j} (T)(x_i)+ \sum \limits_{i=j+1}^{\infty}P_{k_j} (T)(x_i) \Vert \\
& \leq \Vert P_{k_j} (T)(x_j) - y_j \Vert  + \sum \limits_{i=1}^{j-1} \Vert P_{k_j} (T)(x_i) \Vert + \sum \limits_{i=j+1}^{\infty} \Vert P_{k_j} (T)(x_i) \Vert. \\
& \leq j\xi_j + \sum \limits_{i=j+1}^{\infty} \xi_i .
\end{align*}
Thus $\lim \limits_{j \to \infty} \Vert P_{k_j} (T)(x) - y_j \Vert =0$, that is there exists $x \in X \subseteq \mathcal{M}$ such that $ \widehat{Orb(T,x)} \cap \mathcal{M} $ is dense in $\mathcal{M}$. Then $T$ is $\mathcal{M}$ convex-cyclic operator.
\end{proof}

\begin{remark}
\label{remark:N2C}
Notice that in Theorem \ref{th: cri2} the set $X$ is not dense in $\mathcal{M}$.
\end{remark}

Similar to Theorem 3, in previous theorem conditions $(1)$ and $(2)$ are not sufficient for $T$ to be $\mathcal{M}$ convex-cyclic operator, which we show in Example \ref{example: on cri2} in the next section.
\section{Examples}
Before we start this section\footnote{All examples in $\ell^p(\mathbb{N})$ space are based on \cite{Int2}.}, we need some notions. Let $(n_k)_{k=1}^{\infty}$ and $(m_k)_{k=1}^{\infty}$ be increasing sequences of of positive integers such that $ n_k < m_k < n_{k+1} \, $ for all $k.$ In $\ell^p,\, p \geq 1,$ denote by $\{ e_n \}_{n=1}^{\infty}$ the canonical basis for $\ell^p,$ let $B$, $B e_n = e_{n-1},$ be the backward shift operator. Consider the closed linear subspace $\mathcal{M}$ generated by the set $\{ e_j: n_k \leq j \leq m_k, \; k \geq 1 \}$. As an example case we could set $n_k = \{ 1, 3, 9, 27, 81, \cdots \}$ and $m_k = \{ 2, 4, 5, 6, 7, 8, 10, 11, \cdots, 26, \cdots, 80, \cdots \}$, then we would obtain  $\mathcal{M}:= \{e_1, e_2, e_3, \cdots \}= \{ e_j : j \in \mathbb{N} \}. $ 

The following two lemmas show that being $\mathcal{M}$ convex-cyclic operator does not imply $\mathcal{M}$ convex-cyclic transitivity; we have used similar arguments as in \cite{3}. 

\begin{lemma}
\label{Lemma: sup=infty1}
If $\sup_{k \geq 1}(m_k - n_k)=\infty , $\; then $T=2B$ is $\mathcal{M}$ convex-cyclic operator.
\end{lemma}
\begin{proof}
Let $Y= (y_j)_{j=1}^{\infty} \subset c_{00} \cap \mathcal{M}$ be a dense subset of $\mathcal{M}$. Since 
$\sup_{k \geq 1}(m_k - n_k)=\infty , $\; for $y_1$ there exist $k_1$ 
and $ N_1 $ 
such that \,                       $| y_1 | < n_{k_1} < N_1 < N_1 + |y_1| < m_{k_1}.$\,By
induction, it is easy to see that there exist increasing sequences $(N_j)_{j=1}^{\infty}$\, $(k_j)_{j=1}^{\infty}$ \, such that for every fixed $j>1$ we have
\begin{enumerate}
\item 
$| y_j | < n_{k_j} < N_j < N_j + |y_i| < m_{k_j}.$\, for all \, $1 \leq i \leq j,$
\item
$N_j - N_i > n_{k_j} $\, for all \, $1 \leq i < j,$
\end{enumerate}
Let $$X= \bigcup \limits_{j \geq 1} \lbrace S^{N_i} y_j : i \geq j \rbrace .$$ It is clear that $X \subset \mathcal{M}.$ To verify that \,$T, X, Y$\, and $(N_j)_{j=1}^{\infty}\, $ satisfy all conditions of Theorem \ref{th: cri2}, hence that $T$ is $\mathcal{M}$ convex-cyclic operator. Depending on the Example \ref{example: on cri2} conditions 1 and 2 are holds. 
It is enough to check that $$X \subset \bigcap \limits_{j=1}^{\infty} P_{N_j}(T)^{-1}(\mathcal{M}).$$
Let $x= S^{N_i} y_j \in X,$ where $i \geq j.$ For every $l$ consider $P_{N_l}(T)(x).$
If $l > i,$ so
\begin{align*}
 p_{N_l}(T)(x) &= \sum _{\lambda=0}^{N_l} a_\lambda T^\lambda x \\
 &= \sum _{\lambda=0}^{N_{i-1}} a_\lambda T^\lambda x + a_{N_i} T^{N_i} x + \sum _{\lambda=N_{i+1}}^{N_l} a_\lambda T^\lambda x \\
 &= \sum _{\lambda=0}^{N_{i-1}} a_\lambda T^\lambda x + a_{N_i} y_i + 0 \quad \text{Since $T$ is backward shift operator and $x=S^{N_i}y_j$}   \\
 &= \sum _{\lambda=0}^{N_{l-1}} a_\lambda S^{N_i - \lambda} y_j + a_{N_i} y_j.
\end{align*}
Since $N_i - \lambda \geq N_i - N_{i-1} > n_{k_i}$ and $N_i - N_{i-1} < N_i - \lambda + |y_j| < N_i + |y_j| < m_{k_i} $.\\
So $ n_{k_i} < N_i - \lambda  < m_{m_k}$. \\
Hence $p_{N_l}(T)(x) = \sum _{\lambda=0}^{N_{l-1}} a_\lambda S^{N_i - \lambda} y_j + a_{N_i} y_j \in Lin\{e_r: n_{k_i} \leq r \leq m_{k_i}  \} + \mathcal{M} \subseteq \mathcal{M}. $
That is $ p_{N_l}(T)(x) \in \mathcal{M}, $ hence $x \in p_{N_l}(T)^{-1}\left( \mathcal{M} \right).$\\
If $l < i,$ then 
\begin{align*}
p_{N_l}(T)(x) & = \sum _{\lambda=0}^{N_l} a_\lambda T^\lambda x  \\
& = \sum _{\lambda=0}^{N_l} a_\lambda S^{-\lambda} x = \sum _{\lambda=0}^{N_l} a_\lambda S^{N_{i}-\lambda} y_j 
\end{align*}
But $ N_i-\lambda \geq N_i - N_l > n_{k_i},$ and $N_i - N_l \leq N_i - \lambda + |y_l| < N_i +|y_l| < m_{k_i}$.\\
So $ n_{k_i} < N_i - \lambda  < m_{k_i}$ \\
Hence $ p_{N_l}(T)(x) = \sum _{\lambda=0}^{N_l} a_\lambda S^{N_{i}-\lambda} y_j \in Lin\{e_r: n_{k_i} \leq r \leq m_{k_i}  \}  \subset \mathcal{M}, $ hence $x \in p_{N_l}(T)^{-1}\left( \mathcal{M} \right).$\\
If $l = i,$ then
\begin{align*}
p_{N_l}(T)(x) & = \sum _{\lambda=0}^{N_l} a_\lambda T^\lambda x \\
& = \sum _{\lambda=0}^{N_{l-1}} a_\lambda T^{\lambda} x + a_{N_l} T^{N_{l}} x \\
&=  \sum _{\lambda=0}^{N_{l-1}} a_\lambda S^{N_i\lambda} (y_j)+ a_{N_l} S^{N_i-N_l} (y_j) \\
&=  \sum _{\lambda=0}^{N_{l-1}} a_\lambda S^{N_i - \lambda} (y_j)+ a_{N_l}(y_j) .
\end{align*}
Since $N_i - \lambda \geq N_i - N_{l-1} > N_{k_i}$ and $N_i - N_{l-1} \leq N_i - \lambda + |y_l| < N_i + |y_l| < m_{k_i}$. \\
So $ n_{k_i} < N_i - \lambda  < m_{k_i}$. \\
Hence $ p_{N_l}(T)(x) = \sum _{\lambda=0}^{N_{l-1}} a_\lambda S^{N_i - \lambda} (y_j)+ a_{N_l}(y_j) 
Lin\{e_r: n_{k_i} \leq r \leq m_{k_i}  \} + \mathcal{M} \subset \mathcal{M} . $ \\
That is $x \in p_{N_l}(T)^{-1}\left(\mathcal{M} \right).$\\
Consequently, $T$ satisfies all conditions in Theorem \ref{th: cri2}, so $T$ is $\mathcal{M}$ convex-cyclic operator.
\end{proof}

\begin{lemma}
\label{Lemma: sup=infty2}
If $\sup_{k \geq 1}(n_{k+1} - m_k)=\infty , $\; then $T=2B$ is not  $\mathcal{M}$ convex-cyclic transitive.
\end{lemma}

\begin{proof}
Suppose that $T=2B$ is $\mathcal{M}$ convex-cyclic transitive. Let $U$ and $V$ be two non-empty open subsets of $\mathcal{M}$ and suppose that there exists a positive number $m$ such that $U \cap P_m (T)^{-1} V$ contains an open subset $W$ of $\mathcal{M}$. For $x \in W$, there exist $\epsilon > 0 $ such that if  $y \in \mathcal{M}$ and  $|| x - y || < \epsilon$; since $W$ is an open, and then $y \in W$.\\
Since $\sup_{k \geq 1}(n_{k+1} - m_k)=\infty , $ then $ \exists j \in \mathbb{Z}^+ \; \text{such that} \quad n_{j+1}- m_j > m. $\\
Since $|| e_{n_{j+1}} || \leq 1 $ \; and $\frac{1}{2} \epsilon < \epsilon $, \;
$|| \frac{1}{2} \epsilon \; e_{n_{j+1}} || \leq \epsilon . 1 $ \\
So $ ||x- (x+ \frac{1}{2} \epsilon \; e_{n_{j+1}} )|| < \epsilon . 1 $ \\
Now, consider $y := x + e_{n_{j+1}} \; \frac{1}{2} \epsilon \in \mathcal{M} $ and $||x-y||< \epsilon $ since $y \in \mathcal{M}$ so $y \in W \subset U \cap P_m (T)^{-1} V$, and $y \in P_m (T)^{-1} V $\, which means $P_m (T)(y) \in V $\, and \\
 $P_m (T)(y)=P_m (T)\left( x+ \frac{1}{2} \epsilon P_m (T) e_{n_{j+1}} \right) = P_m (T)(x)+ \frac{1}{2} \epsilon P_m (T) e_{n_{j+1}} \in V \subset \mathcal{M}$ \\
Then we get that  $ P_m (T) e_{n_{j+1}} \in \mathcal{M}$ \; which is contradiction, since
\begin{align*}
P_m (T) e_{n_{j+1}} &= \sum_{\lambda=0}^{m} a_{\lambda} T^{\lambda} (e_{n_{j+1}}) \\
 &= \sum_{\lambda=0}^{m} a_{\lambda} e_{n_{j+1}-\lambda} \in \; Lin\{e_r: n_j \leq r \leq m_j \} \subseteq \mathcal{M}. 
\end{align*}
That is $ n_{j+1}- \lambda \leq m_j $ and $ n_{j+1}-m \leq n_{j+1} - \lambda \leq m_j$. Which is contradicts the fact that $ n_{j+1}- m > m_j $.
\end{proof}

\begin{proposition}
\label{proposition: 4.8}
If $\sup_{k \geq 1}(m_k - n_k)=\sup_{k \geq 1}(n_{k+1} - m_k) = \infty , $\; then $T=2B$ is $\mathcal{M}$ convex-cyclic operator.
But $T$ is not  $\mathcal{M}$ convex-cyclic transitive.
\end{proposition}

\begin{example}
Let $S$ be a linear operator on  $\ell^p,\, p \geq 1,$ defined as $Se_n= \frac{1}{2} e_{n+1}.$ Let $\{ n_k \}_{k=1}^{\infty}$ be an increasing sequence of positive integer numbers such that $n_0 = 0$ and $n_{k+1} > 2 \sum_{i=1}^{k} n_i. $ Let $L_0 = \mathcal{M}_0 = Lin \{e_0\}$ be the linear space generated by $e_0,$ 
$L_1 = S^{n_1} \mathcal{M}_0,$ \qquad $\mathcal{M}_1 = \mathcal{M}_0 \oplus L_1,$ and in general, let 
$L_{k+1} = S^{n_{k+1}} \mathcal{M}_k,$ \qquad $\mathcal{M}_{k+1} = \mathcal{M}_k \oplus L_{k+1}.$
Define $$\mathcal{M}=\overline{\underset{k \geq }{\bigcup} \mathcal{M}_k}.$$
It is easy to show by induction that $\mathcal{M}_k \subset Lin\{ e_j : j \leq \sum_{i=0}^{k} n_i \},$ and 
$$L_{k+1} \subset Lin \{ e_j: n_{k+1} \leq j \leq \sum \limits_{i=0}^{k+1} n_i \}.$$
Thus if $x \in \mathcal{M},$ then $x(j)=0$ \quad for all $ \sum \limits_{i=0}^{k} n_i < j < n_{k+1}.$
\\ As a special case, for more clarify $n_0=0$ and since $$n_{k+1} > 2 \sum \limits_{i=0}^{k} n_i , $$
So we can suppose that.
\begin{align*}
n_1 & = n_{0+1}=1 > 2 \sum \limits_{i=0}^{0} n_i = n_0 = 0 \\
n_2 & = n_{1+1}=3 > 2 \sum \limits_{i=0}^{1} n_i = 2(n_0+n_1)=2(0+1) = 2\\
n_3 & = n_{2+1}=9 > 2 \sum \limits_{i=0}^{2} n_i = 2(n_0+n_1+n_2)=2(0+1+3) = 8\\
n_4 & = n_{3+1}=27 > 2 \sum \limits_{i=0}^{3} n_i = 26.\\
& \qquad \qquad \qquad \vdots \\
 So,& \; \{n_k \}^{\infty}_{k=0}  = \{ 0,1,3,9,27, \cdots \}. \\
L_0 & = \mathcal{M}_0= Lin \{ e_0 \} \\
L_1 & = S^{n_1} (\mathcal{M}_0)=S^{1} (\mathcal{M}_0) = S(\alpha e_0)= \frac{\alpha}{2} e_1 \in Lin \{ e_1 \} \\
\mathcal{M}_1& = \mathcal{M}_0 \oplus L_1 = Lin\{e_0\} + Lin\{ e_1 \} = Lin \{ \{e_0\} \cup \{ e_1 \} \} = Lin \{e_0, e_1 \} \\
L_2& = S^{n_2} (\mathcal{M}_1) = S^{3} (\mathcal{M}_1) = S^3 \left( \alpha_1 e_0 + \alpha_2 e_1) \right) = \frac{\alpha_1}{2^3} e_3 + \frac{\alpha_2}{2^4} e_4 \in Lin \{ e_3, e_4 \} \\
 \mathcal{M}_2& = \mathcal{M}_1 \oplus L_2 = Lin\{e_0,e_1 \} + Lin\{ e_3, e_4 \} =  Lin \{e_0, e_1,e_3,e_4 \} \\
 L_3& = S^{n_3} (\mathcal{M}_2) = S^{9} (\mathcal{M}_2) = S^9 \left( \alpha_1 e_0 + \alpha_2 e_1 + \alpha_3 e_3 + \alpha_4 e_4 \right)\\
&  = \frac{\alpha_1}{2^9} e_9 + \frac{\alpha_2}{2^{10}} e_{10} + \frac{\alpha_3}{2^{12}} e_{12} + \frac{\alpha_4}{2^{13}} e_{13}    \in Lin \{e_9, e_{10}, e_{12}, e_{13} \} \\
\mathcal{M}_3 & = \mathcal{M}_2 \oplus L_3 = Lin \{e_0, e_1,e_3,e_4 \} \oplus Lin \{e_9, e_{10}, e_{12}, e_{13} \} \\
&= Lin \{e_0, e_1,e_3,e_4 , e_9, e_{10}, e_{12}, e_{13} \}  \\
& \qquad \qquad \qquad \vdots \\
& \text{and so on for } \, k=4,5, \cdots , \, \text{we see that} \\
L_{1}& =L_{0+1} \subseteq Lin \{ e_j: n_{0+1} \leq j \leq \sum \limits_{i=0}^{0+1} n_i \} = Lin \{ e_1 \} \\
L_{2}& =L_{1+1} \subseteq Lin \{ e_j: n_{1+1} \leq j \leq \sum \limits_{i=0}^{1+1} n_i \} = Lin \{ e_j ; 3 \leq j \leq 4 \} \\
L_{3}& =L_{2+1} \subseteq Lin \{ e_j: n_{2+1}=9 \leq j \leq \sum \limits_{i=0}^{2+1} n_i = 0+1+3+9 \} \\
& = Lin \{ e_j ; 9 \leq j \leq 13 \} \\
Also & \; \mathcal{M}_1  \subset  \{ e_j : j \leq \sum \limits_{i=0}^{1} n_i= 0+ 1 \} \\
 \mathcal{M}_2 & \subset  \{ e_j : j \leq \sum \limits_{i=0}^{2} n_i= 0+ 1+3 =4 \} \\
 \mathcal{M}_3 & \subset  \{ e_j : j \leq \sum \limits_{i=0}^{3} n_i= 0+ 1+3+9 =13 \} 
\end{align*}

Also notice that $ e_2 \notin  \mathcal{M}_2$, and $e_2, e_5, e_6, e_7, e_8, e_{11} \notin \mathcal{M}_3.$ This can be represented by the the following condition\\
if $x \in \mathcal{M}$ and $\mathcal{M}= \overline{\bigcup_{k \geq 0} \mathcal{M}_k}$, then \\
for $k=1,$ \; $x(j)=0$ \; for  $\sum \limits_{i=0}^{k} n_i=1 < j < 3 = n_{k+1}$ \\
for $k=2,$ \; $x(j)=0$ \; for  $\sum \limits_{i=0}^{k} n_i=4 < j < 9 = n_{k+1}$. \\
Thus we can define $\mathcal{M}_k, \mathcal{M},$ and $L_k$ as above, which satisfy all above conditions of Theorem \ref{th: cri2}.
\end{example}

\begin{example}
\label{example: on cri2}
Let $B$ be an operator defined as above. Then the operator $T=2B$ satisfies the first two conditions of Theorem \ref{th: cri2}, but $T$ is not $\mathcal{M}$ convex-cyclic operator.
\end{example}

\begin{proof}
Let $Y= \left( y_j \right)_{j=1}^{\infty} \subset c_{00} \bigcap \mathcal{M}$ be a dense sequence in $\mathcal{M}$. Then there exists an increasing sequence $ \left( k_j \right)_{j=1}^{\infty} $ such that $y_i \in \mathcal{M}_{k_{j}},   i \leq j.$ \\
Let $X := \bigcup \limits_{j \geq 1} \{ S^{n_{k_i}} y_j \; : \; i \geq j \; \} $ \\
Now we will verify that $T$ with $X,Y$ and $\left(P_{n_{k_i}} \right)_{i=1}^{\infty}$; sequence of convex polynomial, satisfies the first two conditions of Theorem \ref{th: cri2}.
\begin{enumerate}
\item[\underline{First}:]
Let $x \in X$ be an arbitrary element. Then we will show that $P_{n_{k_r}}(T)x \to 0. $
Since $x \in X,$ then there exist $j \geq 1$ such that $x := S^{n_k} y_j \;$ for $i \geq j, $ we choose $n_{k_i}$ large enough such that $\dfrac{1}{2^{n_{k_i}}} \to 0.$
\begin{align*}
P_{n_{k_r}}(T)(x) &= P_{n_{k_r}}(T) S^{n_{k_i}} y_j \\
&= P_{n_{k_r}}(T) \dfrac{1}{2^{n_{k_i}}} y_{j+n_{k_i}} \\
&= \sum_{\lambda = 0}^{n_{k_r}} \frac{a_{\lambda}}{2^{n_{k_i}}} T^{\lambda} y_{j+n_{k_i}} \\
&= \frac{1}{2^{n_{k_i}}} \sum_{\lambda = 0}^{j+ n_{k_{i}-1}} a_\lambda 2^\lambda B^\lambda (y_{j+n_{k_i}}) + \frac{1}{2^{n_{k_i}}} \sum_{\lambda = j+ n_{k_{i}}}^{n_{k_r}} a_\lambda 2^\lambda B^\lambda (y_{j+n_{k_i}}) \\
&= \frac{1}{2^{n_{k_i}}} \sum_{\lambda = 0}^{j+ n_{k_{i}}-1} a_\lambda 2^\lambda B^\lambda (y_{j+n_{k_i}}) + 0  \, \\
& \left( \text{Since $B$ is backward shift operator} \right) \\
& \to 0  \qquad \left( \text{Since $n_{k_r}$ is large enough for $i \geq j$.} \right)\\
& \text{So for all} \; x \in X, \quad P_{n_{k_r}}(T)x \to 0.
\end{align*}

\item[\underline{Second}:]
For each $y \in Y \subset c_{00} \bigcap \mathcal{M}, $ we must show that there exists a sequence $\left(x_{k_i} \right)_{i=1}^{\infty} $ in $X := \bigcup \limits_{j \geq 1} \{ S^{n_{k_i}} y_j \; : \; i \geq j \; \} $ such that.\\
$ x_{k_i}(y_j) \to 0 $ and $ P_{n_{k_i}}(T) x_{k_i} \to y. $\\
Since $y \in Y= \left(y_j \right)^{\infty}_{j=1},$ then there exist $j \geq 1$ such that $y=y_j$ and $y_j \in c_{00} \bigcap \mathcal{M}. $ So there exist $k_i$ such that $y_j \in \mathcal{M}_{k_i},$ for $j \leq i.$  But $ \mathcal{M}_{k_i}= \mathcal{M}_{k_i-1} \oplus L_{k_i},$ so there exist $x_{k_i} \in L_{k_i}$ such that $y_j = y_{j-1}+ x_{k_i},$ and $L_{k_i}= S^{n_{k_i}}(\mathcal{M}_{k_i-1})$ and $i \geq j$ so $x_{k_i} \in X.$ Hence, the existence of the sequence $\left(x_{k_i} \right)_{i=1}^{\infty}$ in $X$ is done.\\
$ x_{k_i} = S^{n_{k_i}} (y_j)= \frac{1}{2^{n_{k_i}}} y_{j+n_{k_i}} $ as $k_i \to \infty $ and $(y_j)\in c_{00}$ then $ n_{k_i} \to \infty$ and  $ x_{k_i} \to 0.$\\
To show $P_{n_{k_i}}(T)x_{k_i} \to y.$ Since
\begin{align*}
P_{n_{k_i}}(T)x_{k_i} &= P_{n_{k_i}}(T)S^{n_{k_i}}(y_j)  \\
& = P_{n_{k_i}}(T)y_{j+n_{k_i}} \frac{1}{2^{n_{k_i}}} \\
&=  \sum_{\lambda = 0}^{ n_{k_{i}}} a_\lambda \frac{1}{2^{n_{k_i}}} T^\lambda (y_{j+n_{k_i}}) \\
&=  \sum_{\lambda = 0}^{n_{k_{i}}} \frac{1}{2^{n_{k_i}}} a_\lambda 2^\lambda B^\lambda (y_{j+n_{k_i}}) \\
&= \frac{a_0}{2^{n_{k_i}}}y_{j+n_{k_i}} + \frac{a_1}{2^{n_{k_i-1}}}y_{j+n_{k_i}-1} + \cdots + a_{n_{k_i}} y_j \\
& \text{as} \; k_i \to \infty, \quad then \; P_{n_{k_i}}(T)x_{k_i} \to y_j = y.
\end{align*}
\end{enumerate}
So the first two conditions of Theorem \ref{th: cri2} satisfied.\\
It remains two show that $T$ is not $\mathcal{M}$ convex-cyclic operator. Suppose that $T$ is $\mathcal{M}$ convex-cyclic operator with $\mathcal{M}$ convex-cyclic vector $x \in \mathcal{M}$. For an arbitrary $k_i$, there exists $m$ such that $P_{n_{k_i}}(T)x \in \mathcal{M},$ where
\begin{equation}
\label{eq:1}
m \geq \sum_{\lambda = 0}^{ k_{i}} n_\lambda.
\end{equation}
One can choose $l_i > k_i$, such that
\begin{equation}
\label{eq:2}
n_{l_i +1} -  \sum_{\mu = 0}^{ l_{i}} n_{k_\mu} > 2m,
\end{equation}
this implies that
$\sum_{\mu = 0}^{ l_{i}} n_{\mu} \leq \sum_{\mu = 0}^{ l_{i}} n_{\mu}+m < n_{l_i+1}-m $ so
\begin{equation}
\label{eq:3}
\sum_{\mu = 0}^{ l_{i}} n_{\mu} < n_{l_i+1}-m < n_{l_i+1}-m + \sum_{\mu = 0}^{ k_{i}} n_{\mu}.
\end{equation}
From Equation (\ref{eq:1}) we get that $-m+ \sum_{\mu = 0}^{ k_{i}} n_{\mu} $ has negative value, so
\begin{equation}
\label{eq:4}
\sum_{\mu = 0}^{ l_{i}} n_{\mu} < n_{l_i+1}-m < n_{l_i+1}-m + \sum_{\mu = 0}^{ k_{i}} n_{\mu} < n_{l_i+1}.
\end{equation}
Equation (\ref{eq:4}) gives us $\sum_{\mu = 0}^{ l_{i}} n_{\mu} + m < n_{l_i+1} < n_{l_i+1} + m$ so there exists some positive integer $r$ such that 
\begin{equation}
\label{eq:5}
\sum_{\mu = 0}^{ l_{i}} n_{\mu} + m  \leq r \leq  n_{l_i+1} +m
\end{equation}
that is $\sum_{\mu = 0}^{ l_{i}} n_{\mu} \leq r-m \leq n_{l_i+1}.$ Now from Remark \ref{remark:N2C} if $x \in \mathcal{M},$ then $x(r-m)=0$ for all $\sum_{\mu = 0}^{ l_{i}} n_{\mu} \leq r-m \leq n_{l_i+1}$,
so for $r$ satisfies Equation (\ref{eq:5}),
$$ P_{n_{k_i}}(T)x(r)=  \sum_{\lambda = 0}^{n_{k_{i}}} a_{\lambda}  T^\lambda x(r) =  \sum_{\lambda = 0}^{n_{k_{i}}} a_{\lambda}  2^\lambda x(r-\lambda)=0. $$
By construction of $\mathcal{M},$ it is easy two see that $\mathcal{M} \bigoplus_{q=0}^{\infty} S^{n_{k_q}} \mathcal{M}_{l_1}$ for some increasing sequence $\left( n_{k_q} \right)_{q=0}^{\infty} $ such that $n_{k_0}=0$ and $n_{k_q+1} - n_q - \sum_{\mu = 0}^{ l_{i}} n_{\mu} > 2m , \quad q \geq 0.$\\
Now, for all \quad $n_{k_q} \leq r \leq  n_{k_q} + \sum_{\mu = 0}^{ k_{i}} n_{\mu}  \quad q \geq 1. $
$$ P_{n_{k_i}}(T)x(r)=  \sum_{\lambda = 0}^{n_{k_{i}}} a_{\lambda}  T^\lambda x(r) =  \sum_{\lambda = 0}^{n_{k_{i}}} a_{\lambda}  2^\lambda x(r-\lambda) \in Lin\{e_0,e_1,\cdots,e_{\kappa}  \} \subseteq \mathcal{M} $$
where $0 \leq r- \lambda \leq \sum_{\mu = 0}^{{k_{i}}} n_\mu = \kappa$ \qquad for all $\lambda=0,1,\cdots,n_{k_{q}}$.\\
So we get for some $x \in X, P_{n_{k_i}}(T)x \notin \mathcal{M}$  i.e. $x \in X$ but $x \notin P_{n_{k_i}}(T)^{-1}(M).$
\end{proof}

To explain the above proof numerically let $n_k= \{0,1,3,9,27, \cdots  \}$, \quad $e_4 \in X$ since $2^4 e_4 = S^3(e_1)$ as $3 \geq 1$. We have
\begin{align*}
P_{n_{k_i}}(T)(e_4)&= 2^4 a_0 e_4 + a_1 T(e_4) + a_2 T^2 e_4 + \cdots + a_{n_k} T^{n_k} e_4 \\
&= \alpha_0 e_4 + \alpha_1 e_3 + \alpha_2 e_2 + \alpha_3 e_1 + \alpha_4 e_0 + 0 + 0 + \cdots + 0 \\
& \in Lin \{ e_0,e_1,e_2,e_3,e_4 \} \qquad \textrm{ but } \qquad e_2 \notin \mathcal{M}.
\end{align*}

\begin{example}
\label{example: Last}
Let $\lambda \in \mathbb{C}$ such that $|\lambda| > 1$, and consider $T:= \lambda B$ where $B$ is the backward shift on $\ell^2.$ Let $\mathcal{M}$ be the subspace of $\ell^2$ consisting of all sequences with zeros on the even entries:
$\mathcal{M} = \left\lbrace \{ a_n \}_{n=0}^{\infty} \in \ell^2 : a_{2k} =0 \; \text{for all} \; k \right\rbrace.$ Then $T$ is $\mathcal{M}$ convex-cyclic operator for $\mathcal{M}$.
\end{example}

\begin{proof}
We will apply Theorem \ref{th: cri1} to give an alternative proof. Let $X=Y$ be the subsets of $\mathcal{M}$ consisting of all finite sequences; i.e., those sequences that only have a finite number of non-zero entries: this clearly is a dense subset of $\mathcal{M}$. Let $P_k(T):= T^{2k}$ where $\{ P_k\}_{k>1}$ is a sequence of convex polynomial. \\
Now let us check that conditions $(1), (2)$ and $(3)$ of Theorem \ref{th: cri1} hold.\\
Let $x \in X$. Since $x$ only has finitely many non-zero entries, $P_k(T)x$ will be zero eventually for $k$ large enough. Thus $(1)$ holds. Let $y \in Y$ and define $x_k:= \frac{1}{\lambda^{2k}} S^{2k}y,$ where$S$ is the forward shift operator on $\ell^2$. Each $x_k$ is in $ \mathcal{M}$ since the even entries of $y$ are shifted by $S^{2k}$ into the even entries of $x_k$. We have\\
$|| x_k ||= \frac{1}{|\lambda|^{2k}} ||y||$, and thus it follows that $x_k \to 0,$ since $|\lambda| > 1.$ Also, because \\
$P_k(T)(x_{k}):= T^{2k}(x_{k}) = (\lambda B)^{2k}(x_{k}) = (\lambda B)^{2k} \frac{1}{\lambda^{2k}} S^{2k} y = y,$ we have that condition $(2)$ holds.\\
The fact that condition $(3)$ holds follows from the fact that if a vector has  a zero on all even positions then it will also have a zero entry on all even positions after the application of the backward shift any even number of times. So the $\mathcal{M}$ convex-cyclic operator of $T$ now follows.
\end{proof}

\begin{remark}
There is a relation between $\mathcal{M}$ convex-cyclic operator and invariant subspace, as we see in Theorem \ref{theorem:1} that  $\mathcal{M}$ is invariant for $P_k(T)$ for all $k$, also in Example \ref{example: 1st} is invariant for $P_k(T)$ whenever $k=1$. But the converse is not true, i.e. if $P_k(T)$ for all $k$ is $\mathcal{M}$ convex-cyclic operator then it does not need to $\mathcal{M}$ to be invariant under $P_k(T)$, as in Example \ref{example: Last} when $k=1$. But subspace $\mathcal{M}$ is invariant under $\mathcal{M}$ convex-cyclic operator for $T^{2k}$.
\end{remark}

\section{Some questions}
We end our paper with some open questions on $\mathcal{M}$ convex-cyclic  operators.

\smallskip

As Bourdon and Feldman \cite{q1} proved that for hypercyclic operator that somewhere dense orbits are everywhere dense so it is natural to ask: 

\textbf{Open question 1.} 
As any orbit of $\mathcal{M}$ convex-cyclic operator in $\mathcal{M}$ is somewhere dense does it imply it to be everywhere dense in $\mathcal{M}$? 

\textbf{Open question 2.}
Does there exist $\mathcal{M}$ convex-cyclic operator for $\mathcal{M}$ such that we do not have $P_k(T)(M) \subset M $ for any $k$?

\smallskip

H. Rezaei \cite{A} defined convex-cyclic operator and characterized completely such operators on finite dimensional vector spaces, we can also ask: 

\textbf{Open question 3.}
Is there any $\mathcal{M}$ convex-cyclic operator 
on finite dimensional vector spaces? 

\textbf{Open question 4.}
If $T: \mathcal{H} \to \mathcal{H}$ is $\mathcal{M}$ convex-cyclic  operator over $\mathcal{M}$ on Hilbert space $\mathcal{H}$, then is $T^m$ $\mathcal{M}$ convex-cyclic operator for every integer $m > 1$? \\

\end{document}